\documentclass[draft,a4,reqno,11pt]{amsart}
\usepackage{verbatim}
\usepackage{amssymb,amsmath,amsthm}
\usepackage{amsfonts}
\usepackage{array}

\newtheorem{theorem}{Theorem}[section]

\newtheorem{corollary}{Corollary}[theorem]

\begin{document}

\title[\tiny Multiple zeta stars on 3-2-1 indices]{\small Multiple zeta star values on 3-2-1 indices}

\author[\tiny Kh.~Hessami Pilehrood]{Kh.~Hessami Pilehrood}
\address{The Fields Institute for Research in Mathematical Sciences, 222 College St, Toronto, Ontario M5T 3J1 Canada}
\email{hessamik@gmail.com}

\author{T.~Hessami Pilehrood}
\address{The Fields Institute for Research in Mathematical Sciences, 222 College St, Toronto, Ontario M5T 3J1 Canada}
\email{hessamit@gmail.com}


\begin{abstract}
In 2008, Muneta found explicit evaluation of the multiple zeta star value $\zeta^\star(\{3, 1\}^d)$, and in 2013, Yamamoto proved
a sum formula for multiple zeta star values on 3-2-1 indices. In~this paper, we provide another way of deriving the formulas mentioned above.
It is based on our previous work on generating functions for multiple zeta star values and also on constructions of generating functions for restricted sums
of alternating Euler sums. As a result, the formulas obtained are simpler and computationally more effective than the known ones.
Moreover, we give explicit evaluations of 
$\zeta^\star(\{\{2\}^m, 3, \{2\}^m, 1\bigr\}^d)$ and  $\zeta^\star(\{\{2\}^m, 3, \{2\}^m, 1\}^d, \{2\}^{m+1})$
in two ways. The first is based on computation of product of generating functions, while the second uses properties of Bell polynomials.
\end{abstract}

\maketitle

\section{Introduction.}

The multiple zeta value (MZV) and the multiple zeta star value (MZSV) are defined by
\begin{align}
\zeta(s_1, \ldots, s_r) &=\sum_{k_1> \cdots  > k_r\ge 1}\frac{1}{k_1^{s_1}\cdots k_r^{s_r}},  \label{z}\\[3pt]
 \zeta^\star(s_1, \ldots, s_r) &=\sum_{k_1\ge \cdots  \ge k_r\ge 1}\frac{1}{k_1^{s_1}\cdots k_r^{s_r}} \label{zstar}
\end{align}
for any multi-index ${\bf s}=(s_1, \ldots, s_r)\in{\mathbb N}^r$  with $s_1>1$. By convention, the number
$|{\bf s}| := s_1+\cdots+s_r$ is called the weight, and $r$ the depth (or length) of the multiple zeta (star) value.
The two types of zeta values are expressible in terms of each other via the relations
\begin{equation*}
\zeta^\star({\bf s})=\sum_{\bf p}\zeta({\bf p}) \qquad\text{and}\qquad \zeta({\bf s})=\sum_{\bf p}(-1)^{\sigma({\bf p})}\zeta^\star({\bf p}),
\end{equation*}
where ${\bf p}$ runs through all the indices  of the form $(s_1\circ s_2\circ\cdots \circ s_r)$ with ``$\circ$'' being either the symbol ``,'' or
the sign ``$+$'', and the exponent $\sigma({\bf p})$ is the number of signs ``$+$'' in ${\bf p}$. This implies that the two ${\mathbb Q}$-vector spaces
spanned by the MZVs and by the MZSVs  coincide. 

Special values of (\ref{z}) and (\ref{zstar})  have recently been the subject of a lot of experimentation and several conjectures, many of which were suggested
by numerical calculations. 
The simplest precise evaluations of multiple zeta values are given by
\begin{equation} \label{in_0}
\zeta(\{2\}^d)=\frac{\pi^{2d}}{(2d+1)!} \qquad\text{and}\qquad \zeta^\star(\{2\}^d)=(-1)^{d+1}(2^{2d}-2)\frac{B_{2d}}{(2d)!}\,\pi^{2d},
\end{equation}
where $B_{2d}$ is a  classical Bernoulli number. Here by $\{a_1, \ldots, a_k\}^d$ we mean $d$ successive repetitions of the sequence $\{a_1, \ldots, a_k\}$.
The formulas easily follow from the infinite product of the sine function and Laurent series expansions of  functions $\sin(\pi z)/(\pi z)$ and $\pi z/\sin(\pi z)$,
respectively. 
Formulas (\ref{in_0}) were further extended  by many authors (see, for example, \cite{BBB:1997, Chen:2017JNT, Mu:2008}) to
\begin{equation} \label{in_1}
\zeta(\{2m\}^d)=C_{m,d}\pi^{2md} \qquad\text{and}\qquad \zeta^\star(\{2m\}^d)=C_{m,d}^{\star}\pi^{2md},
\end{equation}
with explicitly given rational constants $C_{m,d}$ and $C_{m,d}^{\star}$.

Another example of arbitrary depth evaluation includes
\begin{equation} \label{in_2}
\zeta(\{3,1\}^d)=\frac{2\pi^{4d}}{(4d+2)!},
\end{equation}
which was conjectured by Zagier \cite{Za:1994} and first proved by Borwein et al.\ \cite{BBB:2001} by using generating functions, namely, by showing that
$$
\sum_{d=0}^{\infty}\zeta(\{3,1\}^d)z^{4d}=\frac{\sin(\frac{1}{2}(1+i)\pi z)}{\frac{1}{2}(1+i)\pi z}\cdot \frac{\sin(\frac{1}{2}(1-i)\pi z)}{\frac{1}{2}(1-i)\pi z}.
$$
Later, a purely combinatorial proof was given in \cite{BBB:1998} based on shuffle properties of iterated integrals.
A multiple zeta star version of formula (\ref{in_2}) was proved by Muneta \cite{Mu:2008}
\begin{equation} \label{in_3}
\zeta^\star(\{3,1\}^d)=\pi^{4d}\sum_{j=0}^d\frac{2}{(4j+2)!}\underset{n_0, n_1\ge 0}{\sum_{n_0+n_1=2(d-j)}}(-1)^{n_1}
\frac{(2^{2n_0}-2)B_{2n_0}}{(2n_0)!}\frac{(2^{2n_1}-2)B_{2n_1}}{(2n_1)!}
\end{equation}
with the help of the identity
\begin{equation} \label{in_4}
\zeta^\star(\{3,1\}^d)=\sum_{j=0}^d\zeta(\{3,1\}^j)\zeta^\star(\{4\}^{d-j}),
\end{equation}
which after substitution of known formulas (\ref{in_1}) and (\ref{in_2}) implies (\ref{in_3}). Muneta's proof of (\ref{in_4})
uses algebraic and combinatorial properties of harmonic (stuffle) product. 

A further generalization of (\ref{in_2}) to a sum of multiple zeta values obtained by inserting blocks of twos of constant total length in the argument string $\{3,1\}^d$, 
$$
Z(d, n):=\underset{a_1, \ldots, a_{2d+1}\ge 0}{\sum_{a_1+\cdots+a_{2d+1}=n}}\zeta(\{2\}^{a_1}, 3, \{2\}^{a_2}, 1, \ldots, 3, \{2\}^{a_{2d}}, 1, \{2\}^{a_{2d+1}}),
$$
was
given by Bowman and Bradley \cite{BB:2002}
\begin{equation} \label{in_5}
Z(d, n)=\binom{n+2d}{n}\frac{\pi^{2n+4d}}{(2d+1)(2n+4d+1)!}.
\end{equation}
It's counterpart for multiple zeta star values
$$
Z^\star(d, n):=\underset{a_1, \ldots, a_{2d+1}\ge 0}{\sum_{a_1+\cdots+a_{2d+1}=n}}\zeta^\star(\{2\}^{a_1}, 3, \{2\}^{a_2}, 1, \ldots, 3, \{2\}^{a_{2d}}, 1, \{2\}^{a_{2d+1}}),
$$
was proved by Yamamomto \cite{Ya:2013} by using generating series of truncated multiple zeta sums. In fact, Yamamoto first proved an identity expressing $Z^\star(d,n)$
in terms of $Z(\cdot,\cdot)$,
$$
Z^\star(d,n)=\underset{j+l+v=n}{\sum_{2m+k+u=2d}}(-1)^{j+k}\binom{k+l}{k}\binom{u+v}{u}Z(m,j)\zeta^\star(\{2\}^{k+l})\zeta^\star(\{2\}^{u+v})
$$
and then by applying Bowman-Bradley formula (\ref{in_5}) and formulas (\ref{in_0}), obtained an explicit evaluation for $Z^\star(d,n)$,
\begin{equation} \label{in_6}
Z^\star(d,n)=\pi^{4d+2n}\underset{j+l+v=n}{\sum_{2m+k+u=2d}}(-1)^{j+k}\binom{k+l}{k}\binom{u+v}{u}\binom{2m+j}{j}\frac{\beta_{k+l}\beta_{u+v}}{(2m+1)(4m+2j+1)!},
\end{equation}
where
$$
\beta_r=(2^{2r}-2)\frac{(-1)^{r-1}B_{2r}}{(2r)!}.
$$
In this paper, we provide another way to explicitly evaluate the values of $\zeta^\star(\{3,1\}^d)$ and $Z^\star(d,n)$, which is self-contained and does not use corresponding results
on multiple zeta values. It is based on generating functions for multiple zeta star values from our paper \cite{THP:2018} and generating functions for restricted sums
of alternating Euler sums. There is a lot of work on restricted sum formulas for multiple zeta values \cite{Chen:2017Med, Chen:2017JNT, Hof:2017, Zhao:2015}
that served as a source of inspiration for us. As a result, the formulas obtained are simpler than those in (\ref{in_3}), (\ref{in_6}), and computationally more effective.
\begin{theorem} \label{T4}
For any non-negative integer $d$, we have
\begin{equation} \label{in_7}
\begin{split}
\zeta^\star(\{3,1\}^d) &= 4\pi^{4d}\cdot\sum_{k=0}^{2d}(-1)^k(4^{k+1}-1)\,\frac{B_{2k+2}}{(2k+2)!}\,\frac{B_{4d-2k}}{(4d-2k)!}, \\[2pt]
\zeta^\star(\{3,1\}^d, 2) &= 4\pi^{4d+2}\cdot\sum_{k=0}^{2d+1}(-1)^k(4^{k+1}-1)\,\frac{B_{2k+2}}{(2k+2)!}\,\frac{B_{4d+2-2k}}{(4d+2-2k)!}.
\end{split}
\end{equation}
\end{theorem}
\noindent For example, if Bernoulli numbers are tabulated (given), the computational complexity of formula (\ref{in_3}) is $O(d^2)$, while for (\ref{in_7}),
it is linear, $O(d)$. Fomulas (\ref{in_7}) follow from the generating function
\begin{equation*}
\sum_{d=0}^{\infty}\zeta^{\star}(\{3,1\}^d)z^{4d}-\sum_{d=0}^{\infty}\zeta^{\star}(\{3,1\}^d,2)z^{4d+2}=  \tanh(\pi z/2)\cdot\cot(\pi z/2)
\end{equation*}
that will be proved in Section \ref{S2} (see Theorem \ref{T3}). 

We also provide explicit evaluation of the sum $Z^\star(d, n)$ and its two sub-sums
depending on whether $a_{2d+1}$ is zero or not.
Let
\begin{equation} \label{z0}
Z^\star_0(d,n)=\underset{a_1, \ldots, a_{2d}\ge 0}{\sum_{a_1+\cdots+a_{2d}=n}}\zeta^\star(\{2\}^{a_1}, 3, \{2\}^{a_2}, 1, \ldots, 3, \{2\}^{a_{2d}}, 1)
\end{equation}
and 
\begin{equation} \label{z1}
Z^\star_1(d,n)=\underset{a_1, \ldots, a_{2d+1}\ge 0}{\sum_{a_1+\cdots+a_{2d+1}=n-1}}\zeta^\star(\{2\}^{a_1}, 3, \{2\}^{a_2}, 1, \ldots, 3, \{2\}^{a_{2d}}, 1, \{2\}^{a_{2d+1}+1}),
\end{equation}
then 
\begin{equation} \label{eq08}
Z^\star(d,n)=Z^\star_0(d,n)+Z^\star_1(d,n).
\end{equation}
\begin{theorem} \label{T11}
For any positive integer $d$ and a non-negative integer $n$, we have
\begin{align*}
Z^\star_0(d,n)&=4\pi^{2n+4d}\sum_{k=0}^{n+2d}\sum_{r=0}^n(-1)^{n+k+r}(4^{k+1}-1)
\binom{k}{r}\binom{n+2d-k}{n-r}\frac{B_{2k+2}}{(2k+2)!}\frac{B_{2n+4d-2k}}{(2n+4d-2k)!}, \\
Z^\star(d,n) &=4\pi^{2n+4d}\sum_{k=0}^{n+2d}\sum_{r=0}^n (-1)^{n+k+r}(4^{k+1}-1)\!\binom{k+1}{r}\!\binom{n+2d-k}{n-r}\frac{B_{2k+2}}{(2k+2)!}\frac{B_{2n+4d-2k}}{(2n+4d-2k)!},
\end{align*}
and for $n\ge 1$,
$$
Z^\star_1(d,n)=4\pi^{2n+4d}\sum_{k=0}^{n+2d}\,\sum_{r=0}^{n-1}(-1)^{n+k+r+1}(4^{k+1}-1)\binom{k}{r}\binom{n+2d-k}{n-1-r}\frac{B_{2k+2}}{(2k+2)!}\frac{B_{2n+4d-2k}}{(2n+4d-2k)!}.
$$
\end{theorem}
\noindent Moreover, we find explicit formulas for
$$
\zeta^\star(\{\{2\}^m, 3, \{2\}^m, 1\bigr\}^d)\qquad\text{and}\qquad  \zeta^\star(\{\{2\}^m, 3, \{2\}^m, 1\}^d, \{2\}^{m+1}),
$$ which seem to be new and
have not appeared in the literature before. We prove that these numbers are rational multiples of $\pi^{4(m+1)d}$ and $\pi^{2(m+1)(2d+1)}$, respectively, and give explicit evaluations
of the rational factors in two ways. The first is based on computation of product of generating functions, while the second uses properties of Bell polynomials. 
\begin{theorem} \label{T7}
For any positive integer $d$ and a non-negative integer $m$,
\begin{equation*}
\begin{split}
\zeta^{\star}&(\left\{\{2\}^{m}, 3, \{2\}^{m}, 1\right\}^d)
= 4^{m+1}\pi^{4d(m+1)} \\[3pt]
&\times\underset{n_0, \ldots, n_{m}\ge 0}{\sum_{n_0+\cdots+n_{m}=2(m+1)d}}
\exp\left(\frac{2\pi i}{m+1}\sum_{k=0}^{m}kn_k\right)\prod_{k=0}^{m}\left(\sum_{l_k=0}^{n_k}(4^{l_k+1}-1)\frac{B_{2l_k+2}}{(2l_k+2)!}\frac{B_{2n_k-2l_k}}{(2n_k-2l_k)!}
\,e^{\frac{\pi il_k}{m+1}}\right)
\end{split}
\end{equation*}
and
\begin{equation*}
\begin{split}
\zeta^{\star}&(\left\{\{2\}^{m}, 3, \{2\}^{m}, 1\right\}^d, \{2\}^{m+1})
= (-1)^m4^{m+1}\pi^{(4d+2)(m+1)} \\[3pt]
&\!\!\!\!\!\times\!\underset{n_0, \ldots, n_{m}\ge 0}{\sum_{n_0+\cdots+n_{m}=(m+1)(2d+1)}}
\!\!\exp\left(\frac{2\pi i}{m+1}\sum_{k=0}^{m}kn_k\right)\prod_{k=0}^{m}\left(\sum_{l_k=0}^{n_k}(4^{l_k+1}-1)\frac{B_{2l_k+2}}{(2l_k+2)!}\frac{B_{2n_k-2l_k}}{(2n_k-2l_k)!}
\,e^{\frac{\pi il_k}{m+1}}\right)\!.
\end{split}
\end{equation*}
\end{theorem}

\vspace{0.2cm}

\begin{theorem} \label{Bell}
For any positive integers $d$ and $m$,
\begin{equation*}
\zeta^{\star}(\bigl\{\{2\}^{m-1}, 3, \{2\}^{m-1}, 1\bigr\}^d)=\pi^{4md}\!\underset{k_1, \ldots, k_d\ge 0}{\sum_{k_1+3k_2+\cdots+(2d-1)k_d=2d}}\,
\prod_{j=1}^d\frac{1}{k_j!}\left(\frac{(2-2^{2m(2j-1)}) B_{2m(2j-1)}}{(2j-1)\cdot (2m(2j-1))!}\right)^{k_j},
\end{equation*}
where the sum is taken over all non-negative integers $k_1, \ldots, k_d$ satisfying $\sum_{j=1}^d(2j-1)k_j=2d$,
and
\begin{equation*}
\begin{split}
\zeta^{\star}&(\bigl\{\{2\}^{m-1}, 3, \{2\}^{m-1}, 1\bigr\}^d, \{2\}^m)=\pi^{2m(2d+1)} \\
&\qquad\qquad\times\underset{k_1, \ldots, k_{d+1}\ge 0}{\sum_{k_1+3k_2+\cdots+(2d+1)k_{d+1}=2d+1}}\,
\prod_{j=1}^{d+1}\frac{1}{k_j!}\left(\frac{(2-2^{2m(2j-1)}) B_{2m(2j-1)}}{(2j-1)\cdot (2m(2j-1))!}\right)^{k_j},
\end{split}
\end{equation*}
where the sum is over all non-negative integers $k_1, \ldots, k_{d+1}$ such that $\sum_{j=1}^{d+1}(2j-1)k_j=2d+1$.
\end{theorem}

\noindent For small values of $d$, $d=1,2$, we easily get the following explicit formulas.
\begin{corollary}
For each positive integer $m$,
$$
\zeta^{\star}(\{2\}^{m-1}, 3, \{2\}^{m-1}, 1)=2\pi^{4m}\left(\frac{(1-2^{2m-1})B_{2m}}{(2m)!}\right)^2,
$$
$$
\zeta^{\star}(\{2\}^{m-1}, 3, \{2\}^{m-1}, 1, \{2\}^m)=\frac{\pi^{6m}}{3}\left(\frac{(2^{2m}-2)^3B_{2m}^3}{2\cdot(2m)!^3}+\frac{(2^{6m}-2)B_{6m}}{(6m)!}\right),
$$
\begin{equation*}
\begin{split}
\zeta^{\star}(\{2\}^{m-1}, 3, \{2\}^{m-1}, 1, \{2\}^{m-1}, 3, \{2\}^{m-1}, 1)&=\frac{(2-2^{2m})B_{2m}\,\pi^{8m}}{3\cdot(2m)!} \\
&\times\left(\frac{(1-2^{2m-1})^3B_{2m}^3}{(2m)!^3}+\frac{(2-2^{6m})B_{6m}}{(6m)!}\right).
\end{split}
\end{equation*}
\end{corollary}
\noindent Note that an MZV version of Theorem \ref{T7}  or Theorem \ref{Bell} is open. A related conjecture was formulated by Borwein, Bradley, and Broadhurst \cite{BBB:1997},
\begin{equation} \label{conj}
\zeta(\{2\}^m, \{3, \{2\}^m, 1, \{2\}^m\}^d) =^{\vspace{-0.1cm}\hspace{-0.2cm}?} \frac{2(m+1)\cdot\pi^{4(m+1)d+2m}}{(2(m+1)(2d+1))!}, 
\end{equation}
and it still remains unresolved. A non-explicit version of (\ref{conj}), 
$$
\zeta(\{2\}^m, \{3, \{2\}^m, 1, \{2\}^m\}^d)\in {\mathbb Q}\cdot\pi^{4(m+1)d+2m},
$$
was proved by Charlton \cite{Ch:2015} by using motivic zeta values.

The paper is organized as follows. In Section \ref{S2}, we derive generating functions for restricted sums of alternating Euler sums and prove Theorem \ref{T4}.
In Section \ref{S3}, we find generating functions for the sequences $\zeta^\star(\{\{2\}^m, 3, \{2\}^m, 1\}^d)$ and $\zeta^\star(\{\{2\}^m, 3, \{2\}^m, 1\}^d, \{2\}^{m+1})$, 
and prove Theorem \ref{T7} and Theorem \ref{Bell}. In Section \ref{S4}, we prove the sum formulas for multiple zeta star values on 3-2-1 indices.

\section{Explicit evaluation of $\zeta^{\star}(\{3,1\}^d)$ and $\zeta^{\star}(\{3,1\}^d, 2)$.} \label{S2}

We consider alternating Euler sums (or alternating multiple zeta values) defined by
$$
\zeta({\bf s}; {\mathbf{\varepsilon}})=\zeta(s_1, \ldots, s_r; \varepsilon_1, \ldots, \varepsilon_r)=\sum_{k_1>k_2>\cdots>k_r\ge 1}
\frac{\varepsilon_1^{k_1}\cdots\varepsilon_r^{k_r}}{k_1^{s_1}\cdots k_r^{s_r}}
$$
for all positive integers $s_1,\ldots, s_r$ and $\varepsilon_j\in\{-1,1\}$, $j=1,\ldots,r$, with $(s_1, \varepsilon_1)\ne (1,1)$ in order for the series to converge.
We will also use the signed number notation for the multi-index ${\bf s}$ associated with $\zeta({\bf s}; {\mathbf{\varepsilon}})$ by writing
$s_j$ for the $j$-th component of ${\bf s}$ if $\varepsilon_j=1$, and $\overline{s}_j$ if $\varepsilon_j=-1$. For example,
$$
\zeta(\overline{s}_1, s_2, s_3, \overline{s}_4, \overline{s}_5)=\zeta(s_1, s_2, s_3, s_4, s_5; -1, 1, 1, -1, -1).
$$
As for the multiple zeta (star) values, we call the number $|{\bf s}|=s_1+\cdots+s_r$  the weight, and $r$  the depth (or length) of the alternating Euler sum
$\zeta({\bf s}; {\mathbf{\varepsilon}})$.
For a special type of alternating zeta values of the form 
$$
\zeta(ms_1, ms_2, \ldots, ms_r; (-1)^{s_1}, (-1)^{s_2}, \ldots, (-1)^{s_r}),
$$ 
with positive integers $s_1, \ldots, s_r$,
we define the sum
of all such values of depth $r$ and weight $mn$ with arguments multiples of $m$ by
\begin{equation} \label{aa}
A(m,n,r)=\sum_{s_1+\cdots+s_r=n}\zeta(ms_1, ms_2, \ldots, ms_r; (-1)^{s_1}, (-1)^{s_2}, \ldots, (-1)^{s_r}).
\end{equation}
We will need the following theorem, the proof of which was inspired by \cite[Theorem 5.1]{Chen:2017JNT}.
\begin{theorem} \label{T1}
For $\lambda\in\mathbb{R}$, $z\in{\mathbb C}$, $|z|<1$, we have
\begin{equation}
\label{eq01}
\sum_{n=0}^{\infty}\left(\sum_{r=1}^n A(m,n,r)(1+\lambda)^r\right)z^{mn}
=\prod_{k=1}^{\infty}\left(1+\lambda\frac{(-1)^kz^m}{k^m}\right)\left(1-\frac{(-1)^kz^m}{k^m}\right)^{-1}.
\end{equation}
\end{theorem}
\begin{proof} Let the right hand side of (\ref{eq01}) be $F(\lambda, z)$. Then we have
\begin{equation*}
\begin{split}
F(\lambda, z)&=\prod_{k=1}^{\infty}\left[\left(1+\lambda\frac{(-1)^kz^m}{k^m}\right)\cdot\sum_{s=0}^{\infty}\frac{(-1)^{ks}z^{ms}}{k^{ms}}\right] \\
&=\prod_{k=1}^{\infty}\left[\sum_{s=0}^{\infty}\frac{(-1)^{ks}z^{ms}}{k^{ms}} +\lambda\sum_{s=0}^{\infty}\frac{(-1)^{k(s+1)}z^{m(s+1)}}{k^{m(s+1)}}\right] \\
&=\prod_{k=1}^{\infty}\left[\sum_{s=0}^{\infty}\frac{(-1)^{ks}z^{ms}}{k^{ms}} +\lambda\sum_{s=1}^{\infty}\frac{(-1)^{ks)}z^{ms}}{k^{ms}}\right] \\
&=\prod_{k=1}^{\infty}\left[1+(\lambda+1)\sum_{s=1}^{\infty}\frac{(-1)^{ks}z^{ms}}{k^{ms}}\right] \\
&=1+\sum_{r\ge 1}\sum_{s_1, \ldots, s_r\ge 1}(\lambda+1)^r\sum_{k_1>\cdots >k_r\ge 1}\frac{(-1)^{k_1s_1}\cdots (-1)^{k_rs_r}}{k_1^{ms_1}
\cdots k_r^{ms_r}}\cdot z^{m(s_1+\cdots+s_r)} \\
&=1+\sum_{n=1}^{\infty} z^{mn}\Bigl(\sum_{r=1}^n(\lambda+1)^r A(m,n,r)\Bigr),
\end{split}
\end{equation*}
and the statement follows.
\end{proof}
\noindent To express the values of $\zeta^{\star}(\{3,1\}^d)$ and $\zeta^{\star}(\{3,1\}^d, 2)$ in terms of (\ref{aa}), we will use  \cite[Theorem 1.3 and Theorem 1.6]{THP:2018}.
\begin{theorem}  \label{T2}
For any positive integer $d$, we have
\begin{align*}
\zeta^{\star}(\{3,1\}^d) &= \sum_{r=1}^{2d} 2^r\cdot A(2,2d,r) \\[2pt]
\zeta^{\star}(\{3,1\}^d, 2) &= -\sum_{r=1}^{2d+1} 2^r\cdot A(2,2d+1,r).
\end{align*} 
\end{theorem}
\begin{proof} From  \cite[Theorem 1.3]{THP:2018} we have
$$
\zeta^{\star}(\{3,1\}^d)=\sum_{k_1\ge\cdots \ge k_{2d}\ge 1}\prod_{j=1}^{2d}\frac{(-1)^{k_j} 2^{\Delta(k_{j-1}, k_j)}}{k_j^2},
$$
where $k_0=0$ and 
\begin{equation*}
\Delta(a,b)=\begin{cases}
0, \qquad\text{if}\quad a=b; \\
1, \qquad\text{else}.
\end{cases}
\end{equation*}
Rewriting it in the form
$$
\zeta^{\star}(\{3,1\}^d)=2\sum_{k_1\ge k_2\ge\cdots \ge k_{2d}\ge 1}\frac{(-1)^{k_1+k_2+\cdots +k_{2d}}}{k_1^2k_2^2\cdots k_{2d}^2}\cdot 2^{\sum_{j=2}^{2d}\Delta(k_{j-1}, k_j)}
$$
and then grouping powers of two, we get
$$
\zeta^{\star}(\{3,1\}^d)=\sum_{r=1}^{2d}2^r \, \underset{s_1,\ldots, s_r\ge 1}{\sum_{s_1+\cdots+s_r=2d}}\,\,\,\sum_{k_1>k_2>\cdots k_r\ge 1}
\frac{(-1)^{s_1k_1}(-1)^{s_2k_2}\cdots (-1)^{s_rk_r}}{k_1^{2s_1}k_2^{2s_2}\cdots k_r^{2s_r}},
$$
which is exactly the required formula.

Similarly, from  \cite[Theorem 1.6]{THP:2018} we have
$$
\zeta^{\star}(\{3,1\}^d, 2)= -\sum_{k_0\ge k_1\ge\cdots \ge k_{2d}\ge 1}\prod_{j=0}^{2d}\frac{(-1)^{k_j} 2^{\Delta(k_{j-1}, k_j)}}{k_j^2},
$$
where $k_{-1}=0$. Regrouping in powers of two, we obtain
\begin{equation*}
\begin{split}
\zeta^{\star}(\{3,1\}^d, 2)&=-2\sum_{k_0\ge k_1\ge k_2\ge\cdots \ge k_{2d}\ge 1}\frac{(-1)^{k_0+k_1+\cdots +k_{2d}}}{k_0^2k_1^2\cdots k_{2d}^2}\cdot 2^{\sum_{j=1}^{2d}\Delta(k_{j-1}, k_j)} \\[5pt]
&=-\sum_{r=1}^{2d+1}2^r \, \underset{s_1,\ldots, s_r\ge 1}{\sum_{s_1+\cdots+s_r=2d+1}}\,\,\,\sum_{k_1>k_2>\cdots k_r\ge 1}
\frac{(-1)^{s_1k_1}(-1)^{s_2k_2}\cdots (-1)^{s_rk_r}}{k_1^{2s_1}k_2^{2s_2}\cdots k_r^{2s_r}} \\[5pt]
&= -\sum_{r=1}^{2d} 2^r\cdot A(2, 2d+1, r),
\end{split}
\end{equation*}
and the theorem follows.
\end{proof}
\begin{theorem} \label{T3}
For $z\in{\mathbb C}$, $|z|<1$, we have the following generating function:
\begin{equation*}
\sum_{d=0}^{\infty}\zeta^{\star}(\{3,1\}^d)z^{4d}-\sum_{d=0}^{\infty}\zeta^{\star}(\{3,1\}^d,2)z^{4d+2}=  \tanh(\pi z/2)\cdot\cot(\pi z/2).
\end{equation*}
\end{theorem}
\begin{proof}  By Theorem \ref{T1} and Theorem \ref{T2}, we have
\begin{equation*}
\begin{split}
\sum_{d=0}^{\infty}&\zeta^{\star}(\{3,1\}^d)\,z^{4d} -\sum_{d=0}^{\infty}\zeta^{\star}(\{3,1\}^d,2)\,z^{4d+2} \\
&=\sum_{d=0}^{\infty}z^{4d}\left(\sum_{r=1}^{2d}2^r\cdot A(2,2d,r)\right) + \sum_{d=0}^{\infty}z^{4d+2}\left(\sum_{r=1}^{2d+1}2^r\cdot A(2, 2d+1,r)\right) \\
&=\sum_{n=0}^{\infty}z^{2n}\left(\sum_{r=1}^n2^r\cdot A(2,n,r)\right) \\
&=\prod_{k=1}^{\infty}\left(1+\frac{(-1)^kz^2}{k^2}\right)\left(1-\frac{(-1)^kz^2}{k^2}\right)^{-1}.
\end{split}
\end{equation*}
Now evaluate the infinite products to see that
\begin{equation*}
\begin{split}
\prod_{k=1}^{\infty}\left(1+\frac{(-1)^kz^2}{k^2}\right)&=\prod_{j=1}^{\infty}\left(1+\frac{z^2}{(2j)^2}\right)\cdot\prod_{j=1}^{\infty}\left(1-\frac{z^2}{(2j-1)^2}\right) \\
&=\prod_{j=1}^{\infty}\left(1+\frac{(z/2)^2}{j^2}\right)\cdot\prod_{j=1}^{\infty}\left(1-\frac{z^2}{j^2}\right)\prod_{j=1}^{\infty}\left(1-\frac{z^2}{(2j)^2}\right)^{-1} \\
&=\frac{\sin(\pi iz/2)}{\pi iz/2}\cdot \frac{\sin(\pi z)}{\pi z}\cdot \frac{\pi z/2}{\sin(\pi z/2)} \\
&=\frac{\sinh(\pi z/2)}{\pi z/2}\cdot\cos(\pi z/2).
\end{split}
\end{equation*}
Similarly, we have
\begin{equation*}
\prod_{k=1}^{\infty}\left(1-\frac{(-1)^kz^2}{k^2}\right)^{-1}=\frac{\pi iz/2}{\sinh(\pi iz/2)}\cdot \frac{1}{\cos(\pi iz/2)}=\frac{\pi z/2}{\sin(\pi z/2)\cosh(\pi z/2)}.
\end{equation*}
Therefore, the resulting product is
$$
\frac{\sinh(\pi z/2)}{\pi z/2}\cdot \cos(\pi z/2)\cdot\frac{\pi z/2}{\sin(\pi z/2) \cosh(\pi z/2)} = \tanh(\pi z/2)\cdot\cot(\pi z/2).
$$
\end{proof}
Recall the definition of the Bernoulli numbers, which are rational numbers given by the generating function
$$
\frac{t}{e^t-1}=\sum_{s=0}^{\infty} B_s\frac{t^s}{s!}.
$$
The first few values are 
$$
B_0=1, \quad B_1=-\frac{1}{2}, \quad B_2=\frac{1}{6}, \ldots,
$$
and $B_{2k+1}=0$ for all integers $k\ge 1$.
The values of the Riemann zeta function at even positive integers can be easily evaluated in terms of $\pi$ and Bernoulli numbers:
\begin{equation} \label{eq02}
\zeta(2k)=(-1)^{k+1}\frac{B_{2k}}{2(2k)!}\,(2\pi)^{2k},  \quad\qquad\qquad k=0,1,2,\ldots,
\end{equation}
where by convention, $\zeta(0)=-1/2$.

\subsection{Proof of Theorem \ref{T4}}
\begin{proof} The formulas easily follow from Theorem \ref{T3}. Using the well-known power series expansions
\begin{equation} \label{series}
\begin{split}
\tanh(\pi z/2)&=\frac{4}{\pi z}\sum_{n=1}^{\infty}(-1)^{n+1}(4^n-1)\zeta(2n)\,\frac{z^{2n}}{4^n}, \quad\qquad |z|<1, \\
\cot(\pi z/2)&=-\frac{4}{\pi z} \sum_{n=0}^{\infty}\zeta(2n)\,\frac{z^{2n}}{4^n}, \quad\quad \qquad\qquad\qquad\quad\, |z|<1,
\end{split}
\end{equation}
and multiplying the series, we get
\begin{equation*}
\begin{split}
\tanh(\pi z/2)&\cdot\cot(\pi z/2) = \frac{4}{\pi^2}\sum_{n=0}^{\infty}(-1)^{n+1}(4^{n+1}-1)\zeta(2n+2)\,\frac{z^{2n}}{4^n}\times\sum_{n=0}^{\infty}
\zeta(2n)\,\frac{z^{2n}}{4^n} \\
&=\frac{4}{\pi^2}\sum_{n_1=0}^{\infty}\sum_{n_2=0}^{\infty}(-1)^{n_1+1}(4^{n_1+1}-1)\zeta(2n_1+2)\zeta(2n_2)\,\frac{z^{2(n_1+n_2)}}{4^{n_1+n_2}} \\[2pt]
&=\frac{4}{\pi^2}\sum_{n=0}^{\infty}\frac{z^{2n}}{4^n}\underset{n_1, n_2\ge 0}{\sum_{n_1+n_2=n}}(-1)^{n_1+1}(4^{n_1+1}-1)\zeta(2n_1+2)\zeta(2n_2).
\end{split}
\end{equation*}
Now by Theorem \ref{T3}, extracting coefficients of $z^{4d}$ and $z^{4d+2}$, we obtain
\begin{equation*}
\begin{split}
\zeta^\star(\{3,1\}^d) &= \frac{1}{4^{2d-1}\pi^{2}}\cdot\sum_{k=0}^{2d}(-1)^{k+1}(4^{k+1}-1)\zeta(2k+2)\zeta(4d-2k), \\
\zeta^\star(\{3,1\}^d,2) &= \frac{1}{4^{2d}\pi^{2}}\cdot\sum_{k=0}^{2d+1}(-1)^k(4^{k+1}-1)\zeta(2k+2)\zeta(4d+2-2k).
\end{split}
\end{equation*}
Finally, substituting expression (\ref{eq02}) in terms of $\pi$ and Bernoulli numbers,  we get the desired formulas.  
\end{proof}
\begin{corollary}
For any non-negative integer $d$,
\begin{align*}
\zeta^\star(\{3,1\}^d)&=\quad\sum_{k=0}^{2d}\zeta(\{\overline{2}\}^k)\zeta^\star(\{\overline{2}\}^{2d-k}), \\
\zeta^\star(\{3,1\}^d,2)&=-\sum_{k=0}^{2d+1}\zeta(\{\overline{2}\}^k)\zeta^\star(\{\overline{2}\}^{2d+1-k}).
\end{align*}
\end{corollary}
\begin{proof}
From the proof of Theorem \ref{T3} we have
\begin{equation*}
F(z):=\sum_{d=0}^{\infty}\zeta^\star(\{3,1\}^d)z^{4d}-\sum_{d=0}^{\infty}\zeta^\star(\{3,1\}^d,2)z^{4d+2}=
\prod_{k=1}^{\infty}\left(1+\frac{(-1)^kz^2}{k^2}\right)\left(1-\frac{(-1)^kz^2}{k^2}\right)^{-1}.
\end{equation*}
Since the first infinite product is the generating function for the sequence $\zeta(\{\overline{2}\}^k)$, and the second one
is the generating function for the sequence $\zeta^{\star}(\{\overline{2}\}^k)$, we have
\begin{align*}
\prod_{k=1}^{\infty}\left(1+\frac{(-1)^kz^2}{k^2}\right) &= \sum_{k=0}^{\infty}\zeta(\{\overline{2}\}^k)z^{2k}, \\
\prod_{k=1}^{\infty}\left(1-\frac{(-1)^kz^2}{k^2}\right)^{-1} &= \sum_{k=0}^{\infty}\zeta^{\star}(\{\overline{2}\}^k)z^{2k}.
\end{align*}
Thus,
\begin{equation*}
F(z)=\sum_{k=0}^{\infty}\zeta(\{\overline{2}\}^k)z^{2k}\cdot   \sum_{l=0}^{\infty}\zeta^{\star}(\{\overline{2}\}^l)z^{2l}
=\sum_{n=0}^{\infty}z^{2n}\underset{k, l\ge 0}{\sum_{k+l=n}}\zeta(\{\overline{2}\}^k)\zeta^{\star}(\{\overline{2}\}^l),
\end{equation*}
and the formulas follow by comparing coefficients of powers of $z$.
\end{proof}

\section{Explicit evaluation of $\zeta^{\star}(\{\{2\}^m,3,\{2\}^m,1\}^d)$ and $\zeta^{\star}(\{\{2\}^m,3,\{2\}^m,1\}^d, \{2\}^{m+1})$.} \label{S3}

In this section, we show that Theorem \ref{T2} is easily generalizable, which allows us to obtain explicit expressions for multiple zeta star 
values from the title.

\begin{theorem} \label{T5}
For any non-negative integers $d$ and $m$,
\begin{align*}
\zeta^{\star}(\{\{2\}^m,3,\{2\}^m,1\}^d) &= \quad\sum_{r=1}^{2d}2^r\cdot A(2m+2, 2d, r), \\
\zeta^{\star}(\{\{2\}^m,3,\{2\}^m,1\}^d, \{2\}^{m+1}) &= -\sum_{r=1}^{2d+1}2^r\cdot A(2m+2, 2d+1, r).
\end{align*}
\end{theorem}
\begin{proof}
From \cite[Theorem 1.6]{THP:2018} we have
$$
\zeta^{\star}(\{\{2\}^m,3,\{2\}^m,1\}^d)=\sum_{k_1\ge\cdots\ge k_{2d}\ge 1} \prod_{j=1}^{2d}\frac{(-1)^{k_j} 2^{\Delta(k_{j-1}, k_j)}}{k_j^{2m+2}},
$$
where $k_{-1}=0$. Grouping in powers of 2, we obtain
\begin{equation*}
\begin{split}
\zeta^{\star}(\{\{2\}^m, 3,  & \{2\}^m, 1\}^d) = 2\sum_{k_1\ge\cdots\ge k_{2d}\ge 1}\frac{(-1)^{k_1+\cdots+k_{2d}}\, 2^{\sum_{j=2}^{2d}\Delta(k_{j-1}, k_j)}}{k_1^{2m+2}
\cdots k_{2d}^{2m+2}} \\
&=
\sum_{r=1}^{2d}2^r\underset{s_1, \ldots, s_r\ge 1}{\sum_{s_1+\cdots+s_r=2d}} \, \,\sum_{k_1>\cdots >k_{r}\ge 1} \,
\frac{(-1)^{k_1s_1}\cdots (-1)^{k_rs_r}}{k_1^{(2m+2)s_1}\cdots k_r^{(2m+2)s_r}} \\
&=\sum_{r=1}^{2d}2^r\cdot A(2m+2, 2d, r).
\end{split}
\end{equation*}
Similarly, from  \cite[Theorem 1.6]{THP:2018} we have
\begin{equation*}
\begin{split}
\zeta^{\star}(\{\{2\}^m,3,\{2\}^m,1\}^d, & \{2\}^{m+1}) = -2\sum_{k_0\ge k_1\ge \cdots\ge k_{2d}\ge 1}\, 
\frac{(-1)^{k_0+k_1+\cdots+k_{2d}} \, 2^{\sum_{j=1}^{2d}\Delta(k_{j-1}, k_j)}}{k_0^{2m+2}k_1^{2m+2}\cdots k_{2d}^{2m+2}} \\
&=-\sum_{r=1}^{2d+1}2^r\cdot A(2m+2, 2d+1, r).
\end{split}
\end{equation*}
\end{proof}
\begin{theorem} \label{T6}
For $|z|<1$ and any non-negative integer $m$, we have
\begin{equation*}
\begin{split}
\sum_{d=0}^{\infty}\zeta^{\star}(\{\{2\}^{m}, 3,  \{2\}^{m}, 1\}^d) \,& z^{4(m+1)d} - \sum_{d=0}^{\infty}\zeta^{\star}(\{\{2\}^{m},3,\{2\}^{m},1\}^d,  \{2\}^{m+1})\, 
z^{2(m+1)(2d+1)}  \\
&= -i\prod_{k=0}^{m}\tan\left(\frac{\pi z}{2}\,e^{\frac{\pi i(2k+1)}{2(m+1)}}\right)\cot\biggl(\frac{\pi z}{2}\, e^{\frac{\pi ik}{m+1}}\biggr).
\end{split}
\end{equation*}
\end{theorem}
\begin{proof}
By Theorem \ref{T5} and Theorem \ref{T1}, for any positive integer $m$, we have
\begin{equation} \label{eq03}
\begin{split}
\sum_{d=0}^{\infty}\zeta^{\star}(\{&\{2\}^{m-1},  3,  \{2\}^{m-1}, 1\}^d) \, z^{4md} - \sum_{d=0}^{\infty}\zeta^{\star}(\{\{2\}^{m-1},3,\{2\}^{m-1},1\}^d,  \{2\}^{m})\, 
z^{2m(2d+1)}  \\
&=\sum_{d=0}^{\infty}z^{4md}\sum_{r=1}^{2d}2^r\cdot A(2m, 2d, r) + \sum_{d=0}^{\infty}z^{2m(2d+1)}\sum_{r=1}^{2d+1}2^r\cdot A(2m, 2d+1, r) \\
&=\sum_{n=0}^{\infty}z^{2mn}\sum_{r=1}^n2^r\cdot A(2m,n,r)
=\prod_{k=1}^{\infty}\left(1+\frac{(-1)^k z^{2m}}{k^{2m}}\right)\left(1-\frac{(-1)^k z^{2m}}{k^{2m}}\right)^{-1}.
\end{split}
\end{equation}
Further factoring gives
\begin{equation*}
\begin{split}
\prod_{k=1}^{\infty}\left(1+\frac{(-1)^k z^{2m}}{k^{2m}}\right) & = \prod_{j=1}^{\infty}\left(1+\frac{z^{2m}}{(2j)^{2m}}\right)\left(1-\frac{z^{2m}}{(2j-1)^{2m}}\right) \\
 &=\prod_{j=1}^{\infty}\left(1+\frac{z^{2m}}{(2j)^{2m}}\right)\left(1-\frac{z^{2m}}{j^{2m}}\right)\left(1-\frac{z^{2m}}{(2j)^{2m}}\right)^{-1} \\
 &=\prod_{j=1}^{\infty}\prod_{k=0}^{m-1}\left(1-e^{\frac{\pi i(2k+1)}{m}}\frac{z^2}{4j^2}\right)\left(1-e^{\frac{2\pi ik}{m}}\frac{z^2}{j^2}\right)
 \left(1-e^{\frac{2\pi ik}{m}}\frac{z^2}{4j^2}\right)^{-1}.
\end{split}
\end{equation*}
Similarly,
\begin{equation*}
\begin{split}
\prod_{k=1}^{\infty}\left(1-\frac{(-1)^k z^{2m}}{k^{2m}}\right)^{-1} & = \prod_{j=1}^{\infty}\left(1-\frac{z^{2m}}{(2j)^{2m}}\right)^{-1}
\left(1+\frac{z^{2m}}{(2j-1)^{2m}}\right)^{-1} \\
 &=\prod_{j=1}^{\infty}\left(1-\frac{z^{2m}}{(2j)^{2m}}\right)^{-1}\left(1+\frac{z^{2m}}{j^{2m}}\right)^{-1}\left(1+\frac{z^{2m}}{(2j)^{2m}}\right) \\
 &=\prod_{j=1}^{\infty}\prod_{k=0}^{m-1}\left(1-e^{\frac{2\pi ik}{m}}\frac{z^2}{4j^2}\right)^{-1}\left(1-e^{\frac{\pi i(2k+1)}{m}}\frac{z^2}{j^2}\right)^{-1}
 \left(1-e^{\frac{\pi i(2k+1)}{m}}\frac{z^2}{4j^2}\right).
\end{split}
\end{equation*}
Hence, the right-hand side of (\ref{eq03}) becomes
\begin{equation*}
{\rm RHS}=\prod_{j=1}^{\infty}\prod_{k=0}^{m-1}\left(1-e^{\frac{2\pi ik}{m}}\frac{z^2}{4j^2}\right)^{-2}\left(1-e^{\frac{\pi i(2k+1)}{m}}\frac{z^2}{4j^2}\right)^{2}
 \left(1-e^{\frac{\pi i(2k+1)}{m}}\frac{z^2}{j^2}\right)^{-1}\left(1-e^{\frac{2\pi ik}{m}}\frac{z^2}{j^2}\right).
\end{equation*}
Using the infinite product for the sine function, we get
\begin{equation*}
\begin{split}
{\rm RHS}&=\prod_{k=0}^{m-1}\left(\frac{\sin\left(\frac{\pi z}{2} e^{\frac{\pi ik}{m}}\right)}{\frac{\pi z}{2} e^{\frac{\pi ik}{m}}}\right)^{-2}
\left(\frac{\sin\left(\frac{\pi z}{2} e^{\frac{\pi i(2k+1)}{2m}}\right)}{\frac{\pi z}{2} e^{\frac{\pi i(2k+1)}{2m}}}\right)^{2}
\left(\frac{\sin\left(\pi z e^{\frac{\pi i(2k+1)}{2m}}\right)}{\pi z e^{\frac{\pi i(2k+1)}{2m}}}\right)^{-1}
\left(\frac{\sin\left(\pi z e^{\frac{\pi ik}{m}}\right)}{\pi z e^{\frac{\pi ik}{m}}}\right) \\
&= -i\prod_{k=0}^{m-1}\left[\frac{\left(\sin\left(\frac{\pi z}{2} e^{\frac{\pi ik}{m}}\right)\right)^{-2} \left(\sin\left(\frac{\pi z}{2} e^{\frac{\pi i(2k+1)}{2m}}\right)\right)^2}{2
\sin\left(\frac{\pi z}{2} e^{\frac{\pi i(2k+1)}{2m}}\right) \cos\left(\frac{\pi z}{2} e^{\frac{\pi i(2k+1)}{2m}}\right)}
\cdot 2\sin\left(\frac{\pi z}{2} e^{\frac{\pi ik}{m}}\right) \cos\left(\frac{\pi z}{2} e^{\frac{\pi ik}{m}}\right)\right] \\[2pt]
&= -i\prod_{k=0}^{m-1}\frac{\sin\left(\frac{\pi z}{2} e^{\frac{\pi i(2k+1)}{2m}}\right)}{\sin\left(\frac{\pi z}{2} e^{\frac{\pi ik}{m}}\right)}\cdot
\frac{\cos\left(\frac{\pi z}{2} e^{\frac{\pi ik}{m}}\right)}{\cos\left(\frac{\pi z}{2} e^{\frac{\pi i(2k+1)}{2m}}\right)}
=-i\prod_{k=0}^{m-1}\tan\left(\frac{\pi z}{2} e^{\frac{\pi i(2k+1)}{2m}}\right)\cot\biggl(\frac{\pi z}{2} e^{\frac{\pi ik}{m}}\biggr).
\end{split}
\end{equation*}
Now replacing $m$ by $m+1$, we  conclude the proof.
\end{proof}

\subsection{Proof of Theorem \ref{T7}}
\begin{proof}
By Theorem \ref{T6} ,   the formulas follow from the power series expansion of the product
$$
\Pi:=-i\prod_{k=0}^m\tan\left(\frac{\pi z}{2}\,e^{\frac{\pi i(2k+1)}{2(m+1)}}\right)\cot\left(\frac{\pi z}{2}\,e^{\frac{\pi ik}{m+1}}\right).
$$
Using the fact that $\tan(z)=-i\tanh(iz)$ and applying formulas (\ref{series}), we get
\begin{equation*}
\begin{split}
\tan&\left(\frac{\pi z}{2}\,e^{\frac{\pi i(2k+1)}{2(m+1)}}\right)\cdot\cot\left(\frac{\pi z}{2}\,e^{\frac{\pi ik}{m+1}}\right) \\
&\qquad\qquad=-\frac{4}{\pi^2}\, e^{\frac{\pi i}{2(m+1)}}\sum_{n=0}^{\infty}
\frac{4^{n+1}-1}{4^n}\zeta(2n+2)e^{\frac{\pi i(2k+1)n}{m+1}}z^{2n}\cdot\sum_{n=0}^{\infty}\frac{\zeta(2n)}{4^n}\,e^{\frac{2\pi ikn}{m+1}}z^{2n} \\
&\qquad\qquad=-\frac{4}{\pi^2}\, e^{\frac{\pi i}{2(m+1)}}\sum_{n=0}^{\infty}
\frac{z^{2n}}{4^n}\,e^{\frac{2\pi ikn}{m+1}}\sum_{l=0}^{n}(4^{l+1}-1)\zeta(2l+2)\zeta(2n-2l) e^{\frac{\pi il}{m+1}}. 
\end{split}
\end{equation*}
Therefore, expanding the product, we have
\begin{equation*}
\begin{split}
\Pi&=\frac{-i (-4)^{m+1}}{\pi^{2(m+1)}}\,e^{\frac{\pi i}{2}}\prod_{k=0}^m\left(\sum_{n=0}^{\infty}\frac{z^{2n}}{4^n}\,e^{\frac{2\pi ikn}{m+1}}\sum_{l=0}^n(4^{l+1}-1)\zeta(2l+2)\zeta(2n-2l)e^{\frac{\pi il}{m+1}}\right) \\
&=\frac{(-1)^{m+1}4^{m+1}}{\pi^{2(m+1)}}\sum_{n_0=0}^{\infty}\ldots\sum_{n_m=0}^{\infty}\frac{z^{2(n_0+\cdots+n_m)}}{4^{n_0+\cdots+n_m}}\exp\left(\frac{2\pi i}{m+1}\sum_{k=0}^mkn_k\right) \\
&\quad\times\prod_{k=0}^m\,\sum_{l_k=0}^{n_k}(4^{l_k+1}-1)\zeta(2l_k+2)\zeta(2n_k-2l_k)e^{\frac{\pi il_k}{m+1}} \\
&=\frac{(-1)^{m+1}4^{m+1}}{\pi^{2(m+1)}}\sum_{n=0}^{\infty}\frac{z^{2n}}{4^n}\underset{n_0, \ldots, n_m\ge 0}{\sum_{n_0+\cdots+n_m=n}}
\exp\left(\frac{2\pi i}{m+1}\sum_{k=0}^mkn_k\right) \\
&\quad\times\prod_{k=0}^m\,\sum_{l_k=0}^{n_k}(4^{l_k+1}-1)\zeta(2l_k+2)\zeta(2n_k-2l_k)e^{\frac{\pi il_k}{m+1}}.
\end{split}
\end{equation*}
Substituting formulas (\ref{eq02}), we obtain
\begin{equation*}
\begin{split}
\Pi=4^{m+1}\sum_{n=0}^{\infty}(-1)^n\pi^{2n}z^{2n}\underset{n_0, \ldots, n_m\ge 0}{\sum_{n_0+\cdots+n_m=n}}e^{\frac{2\pi i}{m+1}\sum\limits_{k=0}^mkn_k}
\prod_{k=0}^m\,\sum_{l_k=0}^{n_k}(4^{l_k+1}-1)e^{\frac{\pi il_k}{m+1}}\frac{B_{2l_k+2}}{(2l_k+2)!}\frac{B_{2n_k-2l_k}}{(2n_k-2l_k)!}.
\end{split}
\end{equation*}
Now comparing coefficients of powers of $z$ for $n=2d(m+1)$ and $n=(2d+1)(m+1)$, by Theorem \ref{T6}, we get the required formulas.
\end{proof}

We can give alternative evaluations of $\zeta^{\star}(\{\{2\}^m, 3, \{2\}^m, 1\}^d)$, $\zeta^{\star}(\{\{2\}^m, 3, \{2\}^m, 1\}^d, \{2\}^{m+1})$ by using Bell polynomials.
The {\it modified Bell polynomials} $P_n(x_1,\ldots, x_n)$ are defined by the generating function \cite{Chen:2017JNT, Coppo} 
$$
\exp\left(\sum_{k=1}^{\infty}x_k \frac{z^k}{k}\right)=\sum_{n=0}^{\infty}P_n(x_1, \ldots, x_n)z^n,
$$
where $P_0=1$ and for $n\ge 1$,
$$
P_n(x_1, \ldots, x_n)=\underset{k_1, \ldots, k_n\ge 0}{\sum_{k_1+2k_2+\cdots+nk_n=n}}\,
\frac{1}{k_1!k_2!\cdots k_n!}\left(\frac{x_1}{1}\right)^{k_1}\left(\frac{x_2}{2}\right)^{k_2}
\cdots\left(\frac{x_n}{n}\right)^{k_n}.
$$

\newpage

\subsection{Proof of Theorem \ref{Bell}}
\begin{proof}
From (\ref{eq03}) we have
\begin{equation*}
\begin{split}
\sum_{d=0}^{\infty}&\zeta^{\star}(\bigl\{\{2\}^{m-1}, 3, \{2\}^{m-1}, 1\bigr\}^d)z^{4md}-\sum_{d=0}^{\infty}\zeta^{\star}(\bigl\{\{2\}^{m-1}, 3, \{2\}^{m-1}, 1\bigr\}^d, \{2\}^m)z^{2m(2d+1)} \\
&\,\,=\exp\left(\sum_{k=1}^{\infty}\log\left(1+\frac{(-1)^kz^{2m}}{k^{2m}}\right)-\sum_{k=1}^{\infty}\log\left(1-\frac{(-1)^kz^{2m}}{k^{2m}}\right)\right) \\[3pt]
&\,\,=\exp\left(\sum_{k=1}^{\infty}\sum_{n=1}^{\infty}\frac{(-1)^{n-1}}{n}\cdot\frac{(-1)^{kn}z^{2mn}}{k^{2mn}} + \sum_{k=1}^{\infty}\sum_{n=1}^{\infty}
\frac{(-1)^{kn}z^{2mn}}{n\cdot k^{2mn}}\right) \\[3pt]
&\,\,=\exp\left(\sum_{n=1}^{\infty}(1-(-1)^n)\cdot\zeta(2mn; (-1)^n)\cdot\frac{z^{2mn}}{n}\right)
=\sum_{n=0}^{\infty} P_n(x_1, \ldots, x_n)z^{2mn},
\end{split}
\end{equation*}
where
\begin{equation*}
x_k=(1-(-1)^k)\cdot\zeta(2mk; (-1)^k)=
\begin{cases}
0, \qquad&\text{if $k$ is even}; \\
2\zeta(\overline{2mk}), \qquad&\text{if $k$ is odd}.
\end{cases}
\end{equation*}
Comparing coefficients of powers of $z$ on both sides, we obtain
$$
\zeta^{\star}(\bigl\{\{2\}^{m-1}, 3, \{2\}^{m-1}, 1\bigr\}^d)=P_{2d}(x_1, \ldots, x_{2d})
$$
and
$$
\zeta^{\star}(\bigl\{\{2\}^{m-1}, 3, \{2\}^{m-1}, 1\bigr\}^d, \{2\}^m)=P_{2d+1}(x_1,\ldots, x_{2d+1}).
$$
Expanding the above expression, we get
\begin{equation*}
\begin{split}
&\zeta^{\star}(\bigl\{\{2\}^{m-1}, 3, \{2\}^{m-1}, 1\bigr\}^d)=P_{2d}(2\zeta(\overline{2m}), 0, 2\zeta(\overline{6m}), 0, \ldots, 2\zeta(\overline{2m(2d-1)}), 0) \\
&=\underset{k_1, k_3, \ldots, k_{2d-1}\ge 0}{\sum_{k_1+3k_3+\cdots+(2d-1)k_{2d-1}=2d}}\,\frac{1}{k_1!k_3!\cdots k_{2d-1}!}\left(\frac{2\zeta(\overline{2m})}{1}\right)^{k_1}\!\!
\left(\frac{2\zeta(\overline{6m})}{3}\right)^{k_3}\!\!\cdots\left(\frac{2\zeta(\overline{2m(2d-1)})}{2d-1}\right)^{k_{2d-1}}\!\!\!\!.
\end{split}
\end{equation*}
Similarly,
\begin{equation*}
\begin{split}
&\zeta^{\star}(\bigl\{\{2\}^{m-1}, 3, \{2\}^{m-1}, 1\bigr\}^d, \{2\}^m)=P_{2d+1}(2\zeta(\overline{2m}), 0, 2\zeta(\overline{6m}), 0, \ldots,  0, 2\zeta(\overline{2m(2d+1)})) \\
&=\underset{k_1, k_3, \ldots, k_{2d+1}\ge 0}{\sum_{k_1+3k_3+\cdots+(2d+1)k_{2d+1}=2d+1}}\,\frac{1}{k_1!k_3!\cdots k_{2d+1}!}\!\left(\frac{2\zeta(\overline{2m})}{1}\right)^{k_1}\!\!
\left(\frac{2\zeta(\overline{6m})}{3}\right)^{k_3}\!\!\cdots\left(\frac{2\zeta(\overline{2m(2d+1)})}{2d-1}\right)^{k_{2d+1}}\!\!\!\!\!\!\!.
\end{split}
\end{equation*}
Finally, using the formula
$$
\zeta(\overline{s})=\sum_{k=1}^{\infty}\frac{(-1)^k}{k^s}=(2^{1-s}-1)\zeta(s)
$$
and applying representations (\ref{eq02}) in terms of $\pi$ and Bernoulli numbers, we get the desired formulas after replacing $k_{2j-1}$ by $k_j$ for each $j$. 
\end{proof}

\section{Sum formulas for multiple zeta star values on 3-2-1 indices.}  \label{S4}

The purpose of this section is to find explicit evaluation of the sum
$$
Z^\star(d,n)=\underset{a_1, \ldots, a_{2d+1}\ge 0}{\sum_{a_1+\cdots+a_{2d+1}=n}}\zeta^\star(\{2\}^{a_1}, 3, \{2\}^{a_2}, 1, \ldots, 3, \{2\}^{a_{2d}},1,\{2\}^{a_{2d+1}})
$$
by applying generating functions from \cite{THP:2018}.
We will split the above sum into two parts $Z^\star_0(d,n)$ and $Z^\star_1(d,n)$ defined in (\ref{z0}) and (\ref{z1}), and then evaluate each of these sub-sums
separately. 
\begin{theorem} \label{T8}
For any positive integer $d$ and any complex $z$ with $|z|<1$, we have
\begin{align*}
\sum_{n=0}^{\infty} Z^\star_0(d,n) z^{2n} &= \quad\sum_{r=1}^{2d}2^r\cdot A_z(2, 2d, r), \\
\sum_{n=0}^{\infty} Z^\star_1(d,n+1) z^{2n} &= -\sum_{r=1}^{2d+1}2^r\cdot A_z(2, 2d+1, r),
\end{align*}
where
$$
A_z(m,n,r)=\underset{s_1, \ldots, s_r\ge 1}{\sum_{s_1+\cdots+s_r=n}}\,\,\sum_{k_1>\cdots k_r\ge 1}\frac{(-1)^{k_1s_1}\cdots (-1)^{k_rs_r}}{(k_1^m-z^m)^{s_1}\cdots (k_r^m-z^m)^{s_r}}.
$$
\end{theorem}
\begin{proof}
From \cite[Theorem 1.3]{THP:2018} we have
$$
\sum_{a_1, \ldots, a_{2d}\ge 0}\zeta^\star(\{2\}^{a_1}, 3, \{2\}^{a_2}, 1, \ldots, 3, \{2\}^{a_{2d}}, 1)\,z_1^{2a_1}\cdots z_{2d}^{2a_{2d}}
=\sum_{k_1\ge\cdots\ge k_{2d}\ge 1}\prod_{j=1}^{2d}\frac{(-1)^{k_j} 2^{\Delta(k_{j-1}, k_j)}}{k_j^2-z_j^2},
$$
where $k_0=0$. Putting $z_1=\cdots=z_{2d}=z$, we get
\begin{equation*}
\sum_{n=0}^{\infty}z^{2n}\underset{a_1, \ldots, a_{2d}\ge 0}{\sum_{a_1+\cdots+a_{2d}=n}}\zeta^\star(\{2\}^{a_1}, 3, \{2\}^{a_2}, 1, \ldots, 3, \{2\}^{a_{2d}}, 1)
=\sum_{k_1\ge\cdots\ge k_{2d}\ge 1}\prod_{j=1}^{2d}\frac{(-1)^{k_j} 2^{\Delta(k_{j-1}, k_j)}}{k_j^2-z^2}.
\end{equation*}
Grouping the terms in powers of 2 in the last sum, we obtain
\begin{equation*}
\sum_{n=0}^{\infty}Z\star_0(d,n) z^{2n}=\sum_{r=1}^{2d}2^r\underset{s_1, \ldots s_r\ge 1}{\sum_{s_1+\cdots+s_r=2d}}\,\, \sum_{k_1>\cdots >k_r\ge 1}
\frac{(-1)^{s_1k_1}\cdots (-1)^{s_rk_r}}{(k_1^2-z^2)^{s_1}\cdots(k_r^2-z^2)^{s_r}} = \sum_{r=1}^{2d}2^r\cdot A_z(2, 2d, r),
\end{equation*}
and the first formula is proved.

Similarly, for evaluating $Z^\star_1(d,n)$, by \cite[Theorem 1.5]{THP:2018}, we have
\begin{equation*}
\begin{split}
\sum_{a_1, \ldots, a_{2d+1}\ge 0}&\zeta^\star(\{2\}^{a_1}, 3, \{2\}^{a_2}, 1, \ldots, 3, \{2\}^{a_{2d}}, 1, \{2\}^{a_{2d+1}+1})\,z_1^{2a_1}\cdots z_{2d+1}^{2a_{2d+1}} \\
&=-\sum_{k_1\ge\cdots\ge k_{2d+1}\ge 1}\prod_{j=1}^{2d+1}\frac{(-1)^{k_j} 2^{\Delta(k_{j-1}, k_j)}}{k_j^2-z_j^2},
\end{split}
\end{equation*}
where $k_0=0$. Putting $z_1=\cdots=z_{2d+1}=z$, we obtain
\begin{equation*}
\begin{split}
\sum_{n=0}^{\infty}z^{2n}\underset{a_1, \ldots, a_{2d+1}\ge 0}{\sum_{a_1+\cdots+a_{2d+1}=n}}&\zeta^\star(\{2\}^{a_1}, 3, \{2\}^{a_2}, 1, \ldots, 3, \{2\}^{a_{2d}}, 1, \{2\}^{a_{2d+1}+1}) \\
&=-\sum_{k_1\ge\cdots\ge k_{2d+1}\ge 1}\prod_{j=1}^{2d+1}\frac{(-1)^{k_j} 2^{\Delta(k_{j-1}, k_j)}}{k_j^2-z^2}.
\end{split}
\end{equation*}
Grouping the terms in powers of 2, we get
\begin{equation*}
\sum_{n=0}^{\infty}Z^\star_1(d,n+1) z^{2n}=-\sum_{r=1}^{2d+1}2^r\!\underset{s_1, \ldots s_r\ge 1}{\sum_{s_1+\cdots+s_r=2d+1}} \sum_{k_1>\cdots >k_r\ge 1}
\prod_{j=1}^r\frac{(-1)^{s_jk_j}}{(k_j^2-z^2)^{s_j}} = -\sum_{r=1}^{2d+1}2^r\cdot A_z(2, 2d+1, r),
\end{equation*}
and the theorem follows.
\end{proof}

\begin{theorem} \label{T9}
For any integer $m$, real $\lambda$, and complex $w, z$ with $|w|<1$, $|z|<1$, we have
\begin{equation} \label{eq05}
\sum_{n=0}^{\infty}\left(\sum_{r=1}^n A_z(m,n,r)(1+\lambda)^r\right)w^{mn}=\prod_{k=1}^{\infty}\left[
\left(1+\frac{\lambda(-1)^kw^m}{k^m-z^m}\right)\left(1-\frac{(-1)^kw^m}{k^m-z^m}\right)^{-1}\right].
\end{equation}
\end{theorem}
\begin{proof}
 Let the right hand side of (\ref{eq05}) be $F(\lambda, z, w)$. Then we have
\begin{equation*}
\begin{split}
F(\lambda, z, w)&=\prod_{k=1}^{\infty}\left[\left(1+\frac{\lambda(-1)^kw^m}{k^m-z^m}\right)\cdot\sum_{s=0}^{\infty}\frac{(-1)^{ks}w^{ms}}{(k^{m}-z^m)^s}\right] \\
&=\prod_{k=1}^{\infty}\left[\sum_{s=0}^{\infty}\frac{(-1)^{ks}w^{ms}}{(k^{m}-z^m)^s} +\lambda\sum_{s=0}^{\infty}\frac{(-1)^{k(s+1)}w^{m(s+1)}}{(k^{m}-z^m)^{s+1}}\right] \\
&=\prod_{k=1}^{\infty}\left[\sum_{s=0}^{\infty}\frac{(-1)^{ks}w^{ms}}{(k^{m}-z^m)^s} +\lambda\sum_{s=1}^{\infty}\frac{(-1)^{ks}w^{ms}}{(k^{m}-z^m)^s}\right] \\
&=\prod_{k=1}^{\infty}\left[1+(\lambda+1)\sum_{s=1}^{\infty}\frac{(-1)^{ks}w^{ms}}{(k^{m}-z^m)^s}\right] \\
&=1+\sum_{r\ge 1}\sum_{s_1, \ldots, s_r\ge 1}(\lambda+1)^r\sum_{k_1>\cdots >k_r\ge 1}\frac{(-1)^{k_1s_1}\cdots (-1)^{k_rs_r}}{(k_1^{m}-z^m)^{s_1}
\cdots (k_r^{m}-z^m)^{s_r}}\cdot w^{m(s_1+\cdots+s_r)} \\
&=1+\sum_{n=1}^{\infty} w^{mn}\Bigl(\sum_{r=1}^n(\lambda+1)^r A_z(m,n,r)\Bigr),
\end{split}
\end{equation*}
and the theorem follows.
\end{proof}

\begin{theorem} \label{T10}
Let $z, w\in{\mathbb C}$, $|z|<1$, $|w|<1$. Then
\begin{equation} \label{eq06.5}
\begin{split}
\sum_{d=0}^{\infty}&\sum_{n=0}^{\infty}Z^\star_0(d,n) z^{2n}w^{4d}-\sum_{d=0}^{\infty}\sum_{n=1}^{\infty}Z^\star_1(d,n)z^{2n-2}w^{4d+2} \\
&=
\frac{\sqrt{z^2+w^2}}{\sqrt{z^2-w^2}}\, \tan\left(\frac{\pi\sqrt{z^2-w^2}}{2}\right)\cot\left(\frac{\pi\sqrt{z^2+w^2}}{2}\right).
\end{split}
\end{equation}
\end{theorem}
\begin{proof}
By Theorems \ref{T8} and \ref{T9}, we have
\begin{equation} \label{eq06}
\begin{split}
\sum_{d=0}^{\infty}&\sum_{n=0}^{\infty}Z^\star_0(d,n) z^{2n}w^{4d}-\sum_{d=0}^{\infty}\sum_{n=1}^{\infty}Z^\star_1(d,n)z^{2n-2}w^{4d+2} \\
&=\sum_{d=0}^{\infty}\left(\sum_{r=1}^{2d}A_z(2, 2d, r)\cdot 2^r\right)w^{4d}+\sum_{d=0}^{\infty}\left(\sum_{r=1}^{2d+1}A_z(2, 2d+1, r)\cdot 2^r\right)w^{4d+2} \\
&=\sum_{n=0}^{\infty}\left(\sum_{r=1}^nA_z(2,n,r)\cdot 2^r\right)w^{2n} \\
&=\prod_{k=1}^{\infty}\left[\left(1+\frac{(-1)^kw^2}{k^2-z^2}\right)\left(1-\frac{(-1)^kw^2}{k^2-z^2}\right)^{-1}\right].
\end{split}
\end{equation}
Hence, the next step is to evaluate the infinite product in (\ref{eq06}). Using the infinite product expansion for the sine function, we have
\begin{equation*}
\begin{split}
\prod_{k=1}^{\infty}\left(1+\frac{(-1)^kw^2}{k^2-z^2}\right)&=\prod_{k=1}^{\infty}\left(1-\frac{z^2-(-1)^kw^2}{k^2}\right)\left(1-\frac{z^2}{k^2}\right)^{-1} \\
&=\prod_{k=1}^{\infty}\left(1-\frac{z^2-w^2}{4k^2}\right)\left(1-\frac{z^2+w^2}{(2k-1)^2}\right)\left(1-\frac{z^2}{k^2}\right)^{-1} \\
&=\prod_{k=1}^{\infty}\left(1-\frac{z^2-w^2}{4k^2}\right)\left(1-\frac{z^2+w^2}{k^2}\right)\left(1-\frac{z^2+w^2}{4k^2}\right)^{-1}\left(1-\frac{z^2}{k^2}\right)^{-1} \\
&=\frac{\sin(\pi\sqrt{z^2-w^2}/2)}{\pi\sqrt{z^2-w^2}/2}\cdot \frac{\sin(\pi\sqrt{z^2+w^2})}{\pi\sqrt{z^2+w^2}} \cdot\frac{\pi\sqrt{z^2+w^2}/2}{\sin(\pi\sqrt{z^2+w^2}/2)}
\cdot\frac{\pi z}{\sin(\pi z)} \\
&=\frac{\pi z}{\sin(\pi z)} \cdot\frac{\sin(\pi\sqrt{z^2-w^2}/2)}{\pi\sqrt{z^2-w^2}/2} \cdot\cos(\pi\sqrt{z^2+w^2}/2).
\end{split}
\end{equation*}
Similarly, we have
\begin{equation*}
\prod_{k=1}^{\infty}\left(1-\frac{(-1)^kw^2}{k^2-z^2}\right)^{-1}=\frac{\sin(\pi z)}{\pi z}\cdot \frac{\pi\sqrt{z^2+w^2}/2}{\sin(\pi\sqrt{z^2+w^2}/2)} \cdot\frac{1}{\cos(\pi\sqrt{z^2-w^2}/2)}.
\end{equation*}
Finally, combining both products, we get the required formula.
\end{proof}

\newpage

\subsection{Proof of Theorem \ref{T11}}
\begin{proof}
The formulas easily follow from Theorem \ref{T10}. Expanding trigonometric functions into power series of $w$ and $z$, we have
\begin{equation*}
\begin{split}
\frac{1}{\sqrt{z^2-w^2}}&\tan\left(\frac{\pi\sqrt{z^2-w^2}}{2}\right)=\frac{1}{\pi}\sum_{k=0}^{\infty}\frac{4^{k+1}-1}{4^k}\,\zeta(2k+2)(z^2-w^2)^k \\
&=\frac{1}{\pi}\sum_{k=0}^{\infty}\frac{4^{k+1}-1}{4^k}\,\zeta(2k+2)\sum_{l=0}^k(-1)^l\binom{k}{l}w^{2l}z^{2(k-l)} \\
&=\frac{1}{\pi}\sum_{l=0}^{\infty}(-1)^lw^{2l}\sum_{k=l}^{\infty}\frac{4^{k+1}-1}{4^k}\,\zeta(2k+2)\binom{k}{l}z^{2(k-l)} \\
&=\frac{1}{\pi}\sum_{l=0}^{\infty}\sum_{r=0}^{\infty}(-1)^lw^{2l}z^{2r}\binom{l+r}{l}\frac{4^{l+r+1}-1}{4^{l+r}}\,\zeta(2l+2r+2).
\end{split}
\end{equation*}
In the same way,
\begin{equation*}
\begin{split}
\sqrt{z^2+w^2}\,&\cot\left(\frac{\pi\sqrt{z^2+w^2}}{2}\right) = -\frac{4}{\pi}\sum_{m=0}^{\infty}\frac{\zeta(2m)}{4^m}(z^2+w^2)^m \\
&=-\frac{4}{\pi}\sum_{m=0}^{\infty}\frac{\zeta(2m)}{4^m}\sum_{s=0}^m\binom{m}{s}w^{2s}z^{2(m-s)} \\
&=-\frac{4}{\pi}\sum_{s=0}^{\infty}w^{2s}\sum_{m=s}^{\infty}\binom{m}{s}\frac{\zeta(2m)}{4^m}\,z^{2(m-s)} \\
&=-\frac{4}{\pi}\sum_{s=0}^{\infty}\sum_{t=0}^{\infty}w^{2s}z^{2t}\binom{s+t}{s}\,\frac{\zeta(2s+2t)}{4^{s+t}}.
\end{split}
\end{equation*}
After multiplying the series, the right-hand side of (\ref{eq06.5}) becomes
\begin{equation*}
\begin{split}
{\rm RHS} &=-\frac{4}{\pi^2}\sum_{l=0}^{\infty}\sum_{r=0}^{\infty}\sum_{s=0}^{\infty}\sum_{t=0}^{\infty}(-1)^lw^{2(l+s)}z^{2(r+t)}
\binom{l+r}{l}\binom{s+t}{s}\frac{4^{l+r+1}-1}{4^{l+r+s+t}} \\
&\times\zeta(2l+2r+2)\zeta(2s+2t) \\
&=-\frac{4}{\pi^2}\sum_{m=0}^{\infty}\sum_{n=0}^{\infty}\sum_{l=0}^m\sum_{r=0}^n(-1)^lw^{2m}z^{2n}\binom{l+r}{l}\binom{m+n-l-r)}{m-l}
\frac{4^{l+r+1}-1}{4^{m+n}} \\ 
&\times\zeta(2l+2r+2)\zeta(2m+2n-2(l+r)) \\
&=-\frac{4}{\pi^2}\sum_{m=0}^{\infty}\sum_{n=0}^{\infty}\sum_{k=0}^{m+n}\sum_{r=0}^n(-1)^{k+r}w^{2m}z^{2n}\binom{k}{r}\binom{m+n-k}{n-r}
\frac{4^{k+1}-1}{4^{m+n}} \\ 
&\times\zeta(2k+2)\zeta(2m+2n-2k).
\end{split}
\end{equation*}
Extracting coefficients of powers $z^{2n}w^{4d}$ and $z^{2n-2}w^{4d+2}$, we get
\begin{equation*}
Z^\star_0(d,n)=-\frac{1}{4^{2d+n-1}\pi^2}\sum_{k=0}^{n+2d}\sum_{r=0}^n (-1)^{k+r}(4^{k+1}-1)\binom{k}{r}\binom{2d+n-k}{n-r}\zeta(2k+2)\zeta(4d+2n-2k)
\end{equation*}
and
\begin{equation*}
Z^\star_1(d,n)=\frac{1}{4^{2d+n-1}\pi^2}\sum_{k=0}^{n+2d}\sum_{r=0}^{n-1} (-1)^{k+r}(4^{k+1}-1)\binom{k}{r}\binom{2d+n-k}{n-1-r}\zeta(2k+2)\zeta(4d+2n-2k).
\end{equation*}
Finally, substituting formulas (\ref{eq02}) in terms of $\pi$ and Bernoulli numbers, we get the first two formulas.
The formula for $Z^\star(d,n)$ follows from (\ref{eq08}) and the observation
\begin{equation*}
\begin{split}
\sum_{r=0}^n(-1)^r&\binom{k}{r}\binom{n+2d-k}{n-r}+\sum_{r=0}^{n-1}(-1)^{r+1}\binom{k}{r}\binom{n+2d-k}{n-1-r} \\
&=\sum_{r=0}^n(-1)^r\binom{k}{r}\binom{n+2d-k}{n-r}+\sum_{r=1}^{n}(-1)^{r}\binom{k}{r-1}\binom{n+2d-k}{n-r} \\
&=\binom{n+2d-k}{n}+\sum_{r=1}^n(-1)^r\binom{n+2d-k}{n-r}\left(\binom{k}{r}+\binom{k}{r-1}\right) \\
&=\sum_{r=0}^n(-1)^r\binom{n+2d-k}{n-r}\binom{k+1}{r}.
\end{split}
\end{equation*}
\end{proof}

\end{document}